\documentclass[12pt,epsfig,amsfonts]{amsart} 
\usepackage{amsmath,amsthm,amssymb,amscd,epsfig,color}
\usepackage{graphicx}
\usepackage{mathrsfs}
\usepackage{lineno}
\usepackage{fancybox}
\usepackage{overpic}
\usepackage{ulem}

\newcommand{\num}{3}

\numberwithin{equation}{section}
\newcommand{\confrac}[2]{%
  \frac{\displaystyle{%
    \strut\hfill{#1}\hfill\;\vrule}}%
      {\displaystyle{%
       \strut\vrule\;\hfill{#2}\hfill}}}%

\makeatletter

    \newcommand\contFrac{\@ifstar{\@contFracStar}{\@contFracNoStar}}

   \def\singleContFrac#1#2{%
        \begin{array}{@{}c@{}}%
            \multicolumn{1}{c|}{#1}%
            \\%
            \hline%
           \multicolumn{1}{|c}{#2}%
        \end{array}%
   }

    \def\@contFracNoStar#1{%
        \mathchoice{
            \@contFracNoStarDisplay@#1//\@nil%
        }{
            \@contFracNoStarInline@#1//\@nil%
        }{
            \@contFracNoStarInline@#1//\@nil%
        }{
            \@contFracNoStarInline@#1//\@nil%
        }%
    }

    \def\@contFracNoStarDisplay@#1//#2\@nil{%
        \@ifmtarg{#2}{%
            #1%
        }{%
            #1+\cfrac{1}{\@contFracNoStarDisplay@#2\@nil}%
        }%
    }

        \def\@contFracNoStarInline@#1//#2\@nil{%
            \@ifmtarg{#2}{%
                #1%
            }{%
                #1 \@@contFracNoStarInline@@#2\@nil%
            }%
        }
        \def\@@contFracNoStarInline@@#1//#2\@nil{%
            \@ifmtarg{#2}{%
                + \singleContFrac{1}{#1}%
            }{%
                + \singleContFrac{1}{#1} \@@contFracNoStarInline@@#2\@nil%
            }%
        }

    \def\@contFracStar#1{%
        \mathchoice{
            \@contFracStarDisplay@#1////\@nil%
        }{
            \@contFracStarInline@#1//\@nil%
        }{
            \@contFracStarInline@#1//\@nil%
        }{
            \@contFracStarInline@#1//\@nil%
        }%
    }

    \def\@contFracStarDisplay@#1//#2//#3\@nil{%
        \@ifmtarg{#2}{%
            #1%
        }{%
            #1 + \cfrac{#2}{\@contFracStarDisplay@#3\@nil}%
        }%
    }

        \def\@contFracStarInline@#1//#2\@nil{%
            \@ifmtarg{#2}{%
                #1%
            }{%
                #1 \@@contFracStarInline@@#2\@nil%
            }%
        }
        \def\@@contFracStarInline@@#1//#2//#3\@nil{%
            \@ifmtarg{#3}{%
                + \singleContFrac{#1}{#2}%
            }{%
                + \singleContFrac{#1}{#2} \@@contFracStarInline@@#3\@nil%
            }%
        }

\theoremstyle{plain}

\newtheorem{thm}{Theorem}[section]
\newtheorem{lem}[thm]{Lemma}

\newtheorem{pro}[thm]{Proposition}
\theoremstyle{definition}

\newtheorem*{prf*}{Proof}

\newtheorem*{pf*}{}
\newtheorem*{lem*}{LemmaA}
\newtheorem*{lm*}{LemmaB}
\newtheorem*{stra*}{Strategy for the proof of main result A}
\def\vph{\phi}
\makeatletter
\@namedef{subjclassname@2020}{%
  \textup{2020} Mathematics Subject Classification}
\makeatother
\title[How many points contain infinitely many homothetic copies]
{How many points contain homothetic copies\\ in their Hurwitz continued fraction expansion?}

\author{Yuto Nakajima and Hiroki Takahasi}

\date{}

\address{Faculty of Science and Engineering, Doshisha University, Kyoto, 610-0394, JAPAN}
\email{yunakaji@mail.doshisha.ac.jp}
\address{Keio Institute of Pure and Applied Sciences (KiPAS), Department of Mathematics,
Keio University, Yokohama,
223-8522, JAPAN} 
\email{hiroki@math.keio.ac.jp}

\begin{document}

\begin{abstract} 
We prove that the set of complex irrationals 
 whose partial quotients in their Hurwitz continued fraction expansion are naturally regarded as subsets of $\mathbb Z^2$ and contain infinitely many homothetic copies of any finite subset of $\mathbb Z^2$
is of Hausdorff dimension $1$. 
Our result provides a clear and concrete example
of multidimensional pattern emergence in number-theoretic expansions.
\end{abstract}

\maketitle

\section{Introduction}
The existence of arithmetic progressions in sets of integers is a profound problem in mathematics.
Szemer\'edi \cite{S75} confirmed a long standing conjecture of Erd\H{o}s and Tur\'an by showing that any subset of $\mathbb N$ with positive upper density contains arithmetic progressions of arbitrary lengths. Furstenberg and Katznelson \cite{FK78} established a multidimensional extension of Szemer\'edi's theorem, replacing arithmetic progressions by homothetic copies of finite subsets of $\mathbb Z^d$, $d\in\mathbb N$.

A natural question is whether these pattern structures 
exist in number-theoretic expansions. 
 For the regular continued fraction expansion, Tong and Wang \cite{TW}  
 proved that 
 the set of irrationals in $(0,1)$ whose 
partial quotients are strictly increasing and 
contain arithmetic progressions of arbitrary lengths and common differences
is of Hausdorff dimension $1/2$. 
Cao and Zhang \cite{ZC} generalized part of the results in \cite{TW} 
  to infinite iterated function systems on compact intervals.

In this paper we focus on
the Hurwitz complex continued fraction. It represents the minimal natural context in which multidimensional combinatorial and fractal phenomena we study appear in an 
explicit and tractable form.
Under the natural identification of $\mathbb C$ with $\mathbb R^2$, the Gaussian field $\mathbb Q(\sqrt{-1})$ and the ring $\mathbb Z(\sqrt{-1})$ of Gaussian integers are identified  with $\mathbb Q^2$ and $\mathbb Z^2$ respectively.
Set \[U=\left\{x\in \mathbb C\colon 
-\frac{1}{2}\le {\rm Re}(x)<\frac{1}{2},\ -\frac{1}{2}\leq {\rm Im}(x)< \frac{1}{2}\right\}.\]
Any $x\in U\setminus\mathbb Q^2$
   has the unique infinite expansion of the form 
\[x=\confrac{1 }{c_{1}(x)} + \confrac{1 }{c_{2}(x)}+ \confrac{1 }{c_{3}(x)}  +\cdots,\]
where $c_n(x)\in\mathbb Z(\sqrt{-1})\setminus\{\pm1,\pm\sqrt{-1}\}$  
for $n\geq1$, 
see \cite[Theorem~6.1]{DN} for example.
This expansion is 
called the
 {\it Hurwitz continued fraction expansion} \cite{H87}. 
 We restrict ourselves to the set 
\[E=\{x\in U\setminus\mathbb Q^2\colon c_{m}(x)\neq c_{n }(x)  \ \text{\rm for all } m,n\in\mathbb N\text{ with }  m\neq n\}.\]
Each $x\in E$  induces an injection 
$n\in \mathbb N\mapsto c_n(x)\in \mathbb Z^2$, and so
$\{c_n(x)\colon n\in\mathbb N\}$ is naturally regarded as an infinite subset of $\mathbb{Z}^2$.
Let $A\subset\mathbb Z^2$. A subset of $\mathbb Z^2$ of the form $v+nA$, $v\in\mathbb Z^2$, $n\in\mathbb N$ 
is called {\it a homothetic copy of $A$}.
Let 
\[H=
\left\{\begin{tabular}{l}
\!\!\!$x\in E\colon \{c_n(x)\colon n\in\mathbb N\}$ contains infinitely many \!\!\!\\
\!\!\!\! homothetic copies of any finite subset of $\mathbb Z^2$ \!\!\!\!\!\!\!\!\!\end{tabular}
\right\}.\]
Let $\dim_{\rm H}$ denote the Hausdorff dimension on the Euclidean space $\mathbb R^2$.
\begin{thm}\label{thmHI}
$\dim_{\rm H}E=\dim_{\rm H}H=1.$
\end{thm}


The restriction to the Hurwitz continued fraction allows us to construct explicit sequences of partial quotients exhibiting the desired patterns.
Using an iterated function system  extracted from the Hurwitz continued fraction, we first 
 construct {\it a seed set} in $E$ of Hausdorff dimension close to $1$ with a controlled structure. By carefully inserting arbitrarily large integer squares, we modify the seed set to ensure that it lies in $H$ and keeps Hausdorff dimension close to $1$. This verifies
the lower bound $\dim_{\rm H}H\geq 1$.
To our knowledge, this type of construction first appeared in \cite{FW01} in the context of symbolic dynamics.
The upper bound $\dim_{\rm H}E\leq 1$ follows from
\cite[Theorem~1.3]{Ger}.



Our result provides a clear and concrete example
of multidimensional pattern emergence in number-theoretic expansions,
illustrating both the combinatorial richness and fractal complexity of the Hurwitz continued fraction.
Theorem~\ref{thmHI} shows that the set 
$H$ is maximal dimensional within the naturally constrained set $E$, which indicates the richness of the phenomenon. 
The method presented here extends to other 
expansions, 
highlighting the broader relevance of constructive approaches in the study of multidimensional pattern formation.

\section{Preliminaries}
 In this section we summarize basic properties of the Hurwitz continued fraction. 

 \subsection{Extracting an infinite iterated function system}\label{basic-def}
 There is a canonical infinite iterated function system 
 that represents the regular continued fraction,  whereas this is not the case for the Hurwitz continued fraction.
However, neglecting finitely many cylinders one can extract an infinite iterated function system from the Hurwitz continued fraction that can be used for the proof of Theorem~\ref{thmHI}.
Put 
\[\mathbb D_{1}=\{(k,\ell)\in\mathbb Z^2\colon k^2+\ell^2\geq 2\}\ \text{  
and }\
\mathbb D_{2}=\{(k,\ell)\in\mathbb Z^2\colon k^2+\ell^2\geq 8\}.\]
For each $i=(k,\ell)\in\mathbb D_1$, 
define a region
 \[U_{i}=
 \left\{x\in\mathbb C\colon -\frac{1}{2}\leq {\rm Re}(x^{-1})-k<\frac{1}{2},\ -\frac{1}{2}\leq {\rm Im}(x^{-1})-\ell<\frac{1}{2} \right\}.\]
It is easy to see that: 

\begin{itemize}  
\item 
$U_{i}\cap \partial U\neq \emptyset$
if $i\in\mathbb D_1\setminus\mathbb D_2$ and
$\overline{U_{i}}\subset U$ if $i\in\mathbb D_2$;
\item  
$\bigcup\{U_{i}\cap U\colon i\in\mathbb D_1\}=U\setminus\{0\}$.
\end{itemize}

 For each $i=(k,\ell)\in\mathbb D_2$ 
 define a map $\phi_{i}\colon  \overline{U}\to\mathbb C$ by
\[\phi_{i}(x)= \frac{1}{x+k+\sqrt{-1}\ell}.\]
Note that $\phi|_{{\rm int}(U)}$ is univalent, $\phi_{i}(U)=U_{i}$ and $\phi_{i}(\overline{U})\subset \overline U$.
For $n\geq2$ and $(\omega_1,\ldots,\omega_n)\in \mathbb D_2^n$,
write $\phi_{\omega_1\cdots \omega_n}=\phi_{\omega_1}\circ\cdots\circ\phi_{\omega_n}.$
Define a map $\Pi\colon \mathbb D_2^{\mathbb N} \to \overline{U}$ by
\[\Pi(\underline{\omega})=\lim_{n\to \infty}\phi_{\omega_1\cdots\omega_n}(0),\]
where $\underline{\omega}=(\omega_1,\omega_2,\ldots).$ The set
$\Lambda=\Pi(\mathbb D_2^{\mathbb N})$
is contained in $U\setminus\mathbb Q^2$. If $x=\Pi(\underline{\omega})$
then $\omega_n=c_n(x)$ for all $n\in\mathbb N$, see \cite{NT2} for details.
We list properties of the iterated function system $\{\phi_i\colon i\in\mathbb D_2\}$.

\begin{figure}
\begin{center}
\includegraphics[height=10cm,width=10cm]{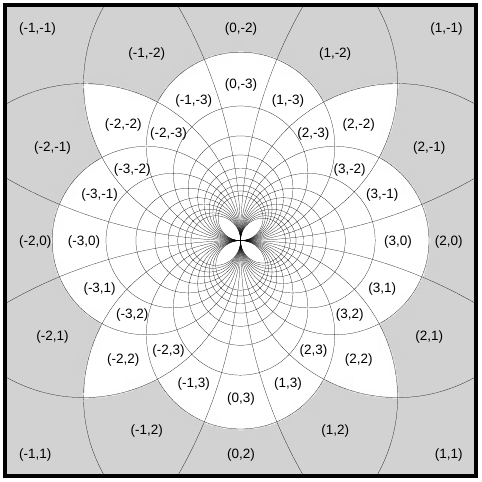}
\caption
{
The shaded region indicates $\bigcup\{U_{i}\cap U\colon i\in\mathbb D_1\setminus\mathbb D_2\}$.}\label{fig2}
\end{center}
\end{figure}

\begin{itemize}

\item[(I)] (open set condition) For all $i,j\in\mathbb D_2$ with $i\neq j$,  \[\phi_{i}({\rm int}\overline{U})\cap \phi_{j}({\rm int}\overline{U})=\emptyset.\]


\item[(II)] (conformality)
For each $i\in\mathbb D_2$, 
$\phi_i|_{U}$ is univalent,  $\phi_{i}(U)=U_{i}$, 
 $\phi_{i}(\overline{U})\subset \overline U$, and has a univalent extension to a connected open set containing $\overline{U}$.

\item[(III)](uniform contraction) 
There exists $\rho\in (0, 1)$ such that for all $i\in \mathbb D_2$, 
\[\sup_{x\in U}|D\phi_{i}(x)|\le \rho,\]
where $D\phi_i(x)$ denotes the complex derivative of $\phi_i$ at $x$.

\item[(IV)] (2-decay) There exist positive constants $C_1$, $C_2$ such that for all $i\in \mathbb D_2$ and  $x\in U$,  \begin{equation}\label{CF-der}\frac{C_1}{|i|^{2 }}\le |D\phi_i(x)|\le \frac{C_2 }{|i|^{2} }.\end{equation}
\item[(V)] (isolation) $\overline{\Lambda}\subset{\rm int}(\overline{U})$.
\end{itemize}
Items (I)--(IV) are easy to check by \cite[Lemma~4.2]{NT2}. Item (V) follows from the observation that $\bigcup\{U_{i}\colon i\in\mathbb D_1\setminus\mathbb D_2\}$ contains a neighborhood of $\partial U$ as in \textsc{Figure}~\ref{fig2}.

We will use two well-known lemmas on univalent functions.
The first one 
follows from the Koebe's distortion theorem
(see e.g., \cite[Theorem~1.4]{CG93}). 
\begin{lem}[{\cite[Lemma~2.2]{NT2}}]\label{distortion-lem}Let $\Omega\subset\mathbb C$ be a region and let $A$ be a compact subset of $\Omega$.
  There exists a constant $K\geq1$ such that for every univalent function $\phi\colon\Omega\to\mathbb C$ and all $x_1$, $x_2\in A$ we have
\[\frac{|D\phi(x_1)|}{|D\phi(x_2)|}\leq K.\]\end{lem}
For $x\in\mathbb C$ and $\delta>0$ let
$B(x, \delta)=\{y\in\mathbb C\colon |x-y|< \delta\}$.
\begin{lem}[{\cite[Lemma~2.3]{NT2}}]\label{conf-lem}
Let $K\geq 1$,
let $\Omega\subset \mathbb C$ be a region and let $\phi\colon \Omega\to \mathbb C$ be univalent. 
For any $x\in \Omega$ and $\delta>0$ such that \[ B(x, \delta)\subset \Omega \ \text{ and }
\ \sup_{x_1,x_2\in  B(x, \delta)}\frac{|D\phi(x_1)|}{|D\phi(x_2)|}\leq K,\]
we have
\[ B\left(\phi(x), \frac{\delta|D\phi(x)|}{3K}\right)\subset \phi( B(x, \delta) ).\]\end{lem}

\subsection{Estimates of Euclidean  diameters of cylinders}\label{diam-sec}
For a set $A\subset\mathbb C$
let ${\rm diam}(A)=\sup\{|x_1-x_2|\colon x_1,x_2\in A\}.$ 
 For $n\in \mathbb N$ and $(c_1,\ldots, c_n)\in  \mathbb D_1^n$, define non-negative integers
 $p_n$ and  $q_n$ by the recursion formulas
\begin{equation}\label{pq}\begin{split}p_{-1}=1,\ p_0=0,\ p_j&=c_jp_{j-1}+p_{j-2} \ \text{ for  } 1\le j\le n,\\
q_{-1}=0,\ q_0=1,\ q_j&=c_jq_{j-1}+q_{j-2} \ \text{ for  } 1\le j\le n.\end{split}\end{equation}
Following \cite[Lemma~2.4]{Ger}, let $\xi\in U$ denote the complex irrational such that
$c_n(\xi)=3+4\sqrt{-1}$ for all $n\geq1$, and set $\gamma=2|\xi|/(|\xi|+1)^2.$

\begin{lem}\label{diam-lem}
For any $n\in \mathbb N,$ $(c_1,\ldots, c_n)\in  \mathbb D_2^n$
 we have
\[\gamma\prod_{j=1}^n\frac{1}{(|c_j|+1)^2}< {\rm diam}( \phi_{c_1,\ldots, c_n}(\overline{U}))<{2}\prod_{j=1}^n\frac{1}{(|c_j|-1)^2}.\]
\end{lem}
\begin{proof}
 Let $n\in \mathbb N$ and $(c_1,\ldots, c_n)\in  \mathbb D_2^n$.
By \cite[Lemma~4.3]{NT2}, ${\rm diam}( \phi_{c_1,\ldots, c_n}(\overline{U}))$ equals the diameter of the set of elements of $U$ which have the finite or infinite Hurwitz continued fraction expansion 
beginning with $c_1,\ldots,c_n$.
 By \cite[Lemma~2.4]{Ger} we have
\begin{equation}\label{diam-lem-eq1}\frac{\gamma}{|q_n|^2}\le {\rm diam}( \phi_{c_1\cdots c_n}(\overline{U}))\le\frac{2}{|q_n|^2}.\end{equation}
By \eqref{pq}
we have
\begin{align*}
|q_n|=|c_nq_{n-1}+q_{n-2}|=|q_{n-1}|\left|c_n+\frac{q_{n-2}}{q_{n-1}}\right|=\cdots=\prod_{j=1}^n\left|c_j+\frac{q_{j-2}}{q_{j-1}}\right|.
\end{align*}
Since $\{|q_n|\}_{n=1}^\infty$ is strictly increasing \cite[p.195]{H87}, for all $1\leq j\leq n$ we have 
\[
|c_j|-1< |c_j|-\left|\frac{q_{j-2}}{q_{j-1}}\right|\le \left|c_j+\frac{q_{j-2}}{q_{j-1}}\right|\le |c_j|
+\left|\frac{q_{j-2}}{q_{j-1}}\right|< |c_j|+1.
\]
Therefore 
\begin{equation}\label{diam-lem-eq2}\prod_{j=1}^n(|c_j|-1)< |q_n|< \prod_{j=1}^n(|c_j|+1).\end{equation}
Combining \eqref{diam-lem-eq1} and \eqref{diam-lem-eq2} yields the desired double inequalities.
\end{proof}

\if0\begin{proof}Fix $z\in \Delta'$ and $\delta>0$ such that $B(z,\delta)\subset \Delta'$.  
By Lemma~\ref{distortion-lem}, there is  $K\geq1$ such that for any $n\in\mathbb N$ and any $(i_1,\ldots, i_n)\in  I^n$,
\begin{equation}\label{diam-eq1}{\rm diam}(\phi_{i_1\cdots i_n}(\Delta'))\leq K\cdot{\rm diam}(\Delta')|D\phi_{i_1\cdots i_n}(z)|.\end{equation}
By Lemma~\ref{conf-lem} we have 
\[ B(\phi_{i_1\cdots i_n}(z), \delta|D\phi_{i_1\cdots i_n}(z)|/(3K))\subset \phi_{i_1\cdots i_n}( B(z, \delta) ),\] 
and in particular
\begin{equation}\label{diam-eq2}{\rm diam}(\phi_{i_1\cdots i_n}(\Delta'))\geq \frac{\delta}{3K}|D\phi_{i_1\cdots i_n}(z)|.\end{equation}

 Since $\{\phi_i\colon i\in I\}$ is 
$2$-decaying, there exists $N\geq1$ such that
for any $i\in I$ with $|i|>N$ we have
\begin{equation}\label{diam-eq3}\frac{C_0^{-1} }{|i|^{2} }\geq\frac{1}{|i|^{2+\theta}}\ \text{ and }\ \frac{C_0 }{|i|^{2} }\leq\frac{1}{|i|^{2-\theta}}.\end{equation}
Combining \eqref{diam-eq1} and the second inequality in \eqref{diam-eq3} yields
\[ {\rm diam}(\phi_{i_1\cdots i_n}(\Delta'))\le K\cdot{\rm diam}(\Delta')C_0^{\#\{1\leq k\leq n\colon |i_k|\leq N\}}\prod_{k=1}^n\frac{1}{|i_k|^{2-\theta}}.\]Combining \eqref{diam-eq2} and the first inequality in \eqref{diam-eq3} yields
\[ {\rm diam}(\phi_{i_1\cdots i_n}(\Delta'))\geq
\frac{\delta}{3K}C_0^{-\#\{1\leq k\leq n\colon |i_k|\leq N\}}\prod_{k=1}^n\frac{1}{|i_k|^{2+\theta}}.\]
Since all the indices $i_1,\ldots,i_n$ are distinct, we have 
$\#\{1\leq k\leq n\colon |i_k|\leq N\}\leq (2\sqrt{N}+1)^2$.
Hence the desired assertion follows.
\end{proof}\fi

\section{On the proof of the main result}
In this section we complete the proof of Theorem~\ref{thmHI}.

\subsection{Dimension estimate of a seed  set}\label{seed-sect}
We introduce the maximum 
norm $\|x\|_\infty=\max\{|{\rm Re}(x)|,|{\rm Im}(x)|\}$ for $x\in \mathbb C$.
Note that $\|x\|_\infty\leq|x|\leq\sqrt{2}\|x\|_\infty$.
We put $t=3$, and introduce a seed set 
\[S=\{x\in U\setminus\mathbb Q^2\colon t^{n}\leq \| c_n(x)\|_{\infty}< 2 t^{n}\ \text{for all}\ n \ge 1\}.\]
For any $x\in S$ the sequence $\{\|c_n(x)\|_{\infty}\}_{n=1}^\infty$ is strictly increasing, and so $S\subset E$. 
\begin{pro}
\label{seed-Prop} 
$\dim_{\rm H}S\geq1$.
\end{pro}

Although this proposition can be proved using the general result in \cite{RU},
we give a self-contained proof.
We set
\[I^{\infty}=\prod_{n=1}^{\infty}I^{(n)},\text{ where }I^{(n)}=\{i\in \mathbb D_2\colon t^{n}\leq \| i\|_{\infty}< 2t^{n}\}.\]
We endow each $I^{(n)}$ with the discrete topology, and $I^{\infty}$ with the product topology.
For a finite word $\omega$ from $\mathbb D_2$, let $|\omega|$ denote its word length. For $\underline{\omega}=(\omega_1,\omega_2,\ldots)\in I^\infty$ and $n\in\mathbb N$, let $\underline{\omega}|_n=\omega_1\cdots\omega_n$.
For a finite word $\tau$ from $\mathbb D_2$, define
\[[\tau]=\{\underline\omega\in I^\infty\colon\underline{\omega}|_{|\tau|}=\tau\}.\]
\begin{lem}\label{unified-lem}
There exist a Borel probability measure $\mu$ on $I^{\infty}$ and a constant $C>0$ such that 
for all $n\in\mathbb N$ and all $\omega\in \prod_{j=1}^n I^{(j)}$, 
\[
\mu([\omega])\le  C\cdot{\rm diam}(\phi_{\omega}(\overline{U})).
\]   
\end{lem}
\begin{proof}
For each $n\in\mathbb N$ let $P_n$
denote the uniform probability distribution on the discrete space $\prod_{j=1}^n I^{(j)}$. 
A direct calculation shows 
$\#I^{(j)}=12t^{2j}-4t^j$ for all $j\geq1$. 
For each  $(\omega_1,\ldots,\omega_n)\in \prod_{j=1}^n I^{(j)}$ we have 
\[P_n(\{(\omega_1,\ldots, \omega_n)\})=\prod_{j=1}^n\frac{1}{\#I^{(j)}}=\prod_{j=1}^n\frac{1}{12 t^{2j}-4 t^j}.\]
By Kolmogorov's extension theorem,
there is a Borel probability measure $\mu$ on $I^{\infty}$ such that for all $n\in\mathbb N$ and all $(\omega_1,\ldots,\omega_n)\in \prod_{j=1}^n I^{(j)}$, 
\begin{equation}\label{mu-ineq}\mu([\omega_1,\ldots, \omega_n])=\prod_{j=1}^n\frac{1}{12 t^{2j}-4 t^j}.\end{equation}
For $1\leq j\leq n$, 
the definition of $I^{(j)}$ gives $|\omega_j|<\sqrt{8} t^j$.  By this and the first inequality in Lemma~\ref{diam-lem},
\begin{equation}\label{mu-ineq3}{\rm diam}(\phi_{\omega}(\overline{U}))\geq
\gamma\prod_{j=1}^n\frac{1}{(|\omega_j|+1)^{2}}>
\gamma\prod_{j=1}^n\frac{1}{(\sqrt{8}t^{j}+1)^{2}}.\end{equation}
Combining \eqref{mu-ineq} and \eqref{mu-ineq3} yields the desired inequality.\end{proof}

Since $\Pi(I^{\infty})\subset S,$ it suffices to show $\dim_{\rm H}\Pi(I^{\infty})\geq1$. 
We will pick a measure $\mu$ as in Lemma~\ref{unified-lem}, and apply the mass distribution principle to 
$\mu\circ(\Pi|_{I^\infty})^{-1}$. To this end we need some preliminary work.
Put \[r_0=\min_{i\in I^{(1)}}|D\phi_i(0)|.\]
For $x\in \Pi(I^\infty)$ and $r\in(0,r_0)$,
let $V(x,r)$ denote the set of finite words $\omega\in \bigcup_{n=1}^{\infty}\prod_{j=1}^nI^{(j)}$ that satisfy the following conditions:
\begin{itemize}
\item[(i)] $B(x, r)\cap \vph_{\omega}(\overline{U})\neq \emptyset;$
\item[(ii)] 
$|D\vph_{\omega}(0)|\geq r$, and
$|D\vph_{\omega a}(0)|< r$
for some $a\in I^{(|\omega|+1)}.$ 
\end{itemize}
Since each $I^{(j)}$, $j\geq1$ is a finite set, 
  $\min\{|\omega|\colon\omega\in V(x,r)\}\to\infty$ as $r\to0$.

Although the collection $\{[\omega]\colon\omega\in V(x,r)\}$ of cylinders covers
$I^\infty\cap \Pi^{-1}B(x, r)$, it may contain redundant cylinders. Hence
we make an adjustment. 
Let $V'(x,r)$ denote the set of $\omega\in V(x,r)$ such that  $[\omega]\subset[\tau]$ holds for some $\tau\in V(x,r)$ with $|\tau|<|\omega|$. Put $ V^*(x,r)=V(x,r)\setminus V'(x,r).$
\begin{lem}\label{adjust}
If
 $x\in \Pi(I^\infty)$ and $r\in(0,r_0)$
then for any $\omega\in V'(x,r)$ there is $\tau\in V^*(x,r)$ such that  $[\omega]\subset[\tau]$ and $|\tau|<|\omega|$.
\end{lem}
\begin{proof}Let $\omega\in V'(x,r).$ There is $\tau_1\in V(x,r)$ such that $[\omega]\subset[\tau_1]$ and $|\tau_1|<|\omega|$. If $\tau_1\in V^*(x,r)$ then we are done. Otherwise, we have $\tau_1\in V'(x,r)$ and there is $\tau_2\in V(x,r)$ such that $[\omega]\subset[\tau_1]\subset[\tau_2]$ and $|\tau_2|<|\tau_1|<|\omega|$. Since $|\omega|$ is finite, repeating this procedure we can find $\tau\in V^*(x,r)$ with the desired property.
\end{proof}

\begin{lem}
\label{claim0}
If
 $x\in \Pi(I^\infty)$ and $r\in(0,r_0)$
then:
\begin{itemize}
\item[(a)] for all $\omega, \tau\in V^*(x,r)$ with $\omega\neq\tau$, $\vph_{\omega}({\rm int}\overline{U})\cap \vph_{\tau}({\rm int}\overline{U})=\emptyset;$
\item[(b)] $I^\infty\cap \Pi^{-1}B(x, r)\subset \bigcup_{\tau\in V^*(x,r)}[\tau].$
\end{itemize}
\end{lem}
\begin{proof}
If $\omega$, $\tau\in W^*(x,r)$ and $\omega\neq\tau$ then   $\omega|_j\neq\tau|_j$ holds for some
 $j\in\min\{|\omega|,|\tau|\}$.
By (I) we have $\vph_{\omega}({\rm int}\overline{U})\cap \vph_{\tau}({\rm int}\overline{U})\subset \vph_{\omega|_j }({\rm int}\overline{U})\cap \vph_{\tau|_j }({\rm int}\overline{U})=\emptyset$,
verifying (a).

Let $\underline\omega\in I^\infty\cap \Pi^{-1}B(x, r)$. 
Then $|D\vph_{\underline\omega|_{n-1}}(0)|\geq r>|D\vph_{\underline\omega|_{n}}(0)|$
holds for some $n\geq2$. Since $\Pi(\underline\omega)\in B(x, r)\cap \vph_{\underline\omega|_{n-1}}(\overline{U})$ we have $\underline\omega|_{n-1}\in V(x,r).$
If $\underline\omega|_{n-1}\in V^*(x,r)$ then $[\underline\omega|_{n-1}]\subset\bigcup_{\tau\in V^*(x,r)}[\tau]$. If $\underline\omega|_{n-1}\in V'(x,r)$ then by Lemma~\ref{adjust}, $\underline\omega|_{n-1}\in[\tau]$ holds for some
 $\tau\in V^*(x,r)$ with $|\tau|<n-1$. Hence we obtain $\underline\omega\in [\tau]$, which 
 verifies (b). 
\end{proof}

\if0\begin{lem}
\label{length}
There exists a constant $C_3\geq1$ such that for all 
 $x\in \Pi(I^\infty)$, $r\in(0,r_0)$ 
and $\omega\in W(x, r),$ we have 
\end{lem}
\begin{proof}
Let $\omega\in V(x,r)$. 
By the second inequality in Lemma~\ref{diam-lem} we have
\[r\le {\rm diam}(\vph_{\omega}(\overline{U}))\le {2}\prod_{j=1}^{|\omega|}\frac{1}{(|\omega_j|-1)^2}\le 2\prod_{j=1}^{|\omega|}\frac{1}{\num^{2(j-1)}}.\]
To deduce the last inequality we have used
$\num^j-1\geq \num^{j-1}$. 
\begin{equation}\label{length}|\omega|\le C\sqrt{\log (1/r)}.\end{equation}
where $C\geq1$ is a uniform constant that is independent of $x$, $r$, $\omega$.
\end{proof}
\fi

\begin{lem}
\label{claim2}
There exist $C>0$ and $r_1\in(0,r_0]$ 
 such that
for  all $r\in(0,r_1)$ and
 $x\in \Pi(I^\infty)$
we have: 
\begin{itemize}
\item[(a)]  ${\rm diam}(\phi_{\omega}(\overline{U}))\le Crt^{2 \sqrt{2\log (1/r)}} $ for any $\omega\in V(x,r)$;
\item[(b)] $\# V^*(x,r)\le  Ct^{2 \sqrt{2\log (1/r)}}.$
\end{itemize}
\end{lem}
\begin{proof}In view of Lemma~\ref{distortion-lem},
fix a constant $K>1$ such that for all $\omega\in \bigcup_{n=1}^\infty\mathbb D_2^n$ we have
\[\sup_{x_1,x_2\in\Delta}\frac{|D\phi_\omega(x_1)|}{|D\phi_\omega(x_2)|}\leq K.\]
Let $\omega\in V(x,r)$.
 By Lemma~\ref{conf-lem}, the second inequality in Lemma~\ref{diam-lem} and
 $t^j\leq|\omega_j|<\sqrt{8} t^j$
for $1\leq j\leq |\omega|$, we have
\[\begin{split}r&\leq|D\phi_\omega(0)|\leq 9K\cdot{\rm diam}(\vph_{\omega}(\overline{U})\\
&\le 18K\prod_{j=1}^{|\omega|}\frac{1}{(|\omega_j|-1)^2}\le 18K\prod_{j=1}^{|\omega|}\frac{1}{t^{2(j-1)}}=\frac{18K}{t^{(|\omega|^2-|\omega|)}}  \leq  \exp\left(-\frac{|\omega|^2}{2}\right),\end{split}\]
where the last inequality holds if $r$ is sufficiently small. 
This gives
\begin{equation}\label{length}|\omega|\le \sqrt{2\log (1/r)}.\end{equation}

Take $a\in I^{(|\omega|+1)}$ such that $|D\vph_{\omega a}(0)|< r.$ 
 We have
\begin{equation}
\begin{split}\label{equar}
r&> | D \vph_{\omega a}(0)|= |D\vph_{\omega}(\vph_a(0))||D\vph_{a}(0)|\\
&\ge \frac{1}{K }  \frac{{\rm diam}(\vph_{\omega}(\overline{U}))}{{\rm diam}(\overline{U})} \min \left\{|D\vph_i(0)|\colon i \in I^{(|\omega|+1)}\right\}\\
&\geq \frac{1}{\sqrt{2}K }  {\rm diam}(\vph_{\omega}(\overline{U}))\frac{C_1}{(\sqrt{8}t^{|\omega|+1})^2}\\
&= \frac{C_1}{72\sqrt{2}K} \frac{{\rm diam}(\vph_{\omega}(\overline{U}))}{t^{2|\omega| }}\geq\frac{C_1 }{72\sqrt{2}K} \frac{{\rm diam}(\vph_{\omega}(\overline{U}))}{t^{2\sqrt{2\log (1/r)}}}.
\end{split}
\end{equation}
The second inequality follows from 
Lemma~\ref{distortion-lem} and the convexity of $\overline{U}$.
The third and fourth ones follow from  \eqref{CF-der} and 
\eqref{length} respectively. The second equality is because $t=3$.

Since $B(x, r)\cap \vph_{\omega}(\overline{U})\neq \emptyset$, \eqref{equar} implies \begin{equation}\label{eq1-leb}\bigcup_{\omega\in V(x,r)}\vph_{\omega}(\overline{U})\subset B\left(x, \left(1+\frac{72\sqrt{2}K}{C_1} t^{2 \sqrt{2\log (1/r)}}\right)r\right).\end{equation}
For each $\omega\in V^*(x,r)$, applying Lemma~\ref{conf-lem} to $\phi_\omega$ and $B(0,1/3)$ we get
$\phi_{\omega}({\rm int} \overline{U})\supset B(\phi_\omega(0),|D\phi(0)|/(9K))$. Using this inclusion and Lemma~\ref{claim0}(a) we get \begin{equation}\label{eq2-leb}\begin{split}{\rm Leb}\left(\bigcup_{\omega\in V^*(x,r)}\vph_{\omega}({\rm int}\overline{U} )\right)&=\sum_{\omega\in V^*(x,r)}{\rm Leb}(\vph_{\omega}({\rm int}\overline{U} ))\\
&\ge \frac{\pi}{81K^2}\sum_{\omega\in V^*(x,r)} |D\vph_{\omega}(0)|^2\\& \ge 
\#V^*(x,r)\frac{\pi r^2}{81K^2},\end{split}\end{equation}
where ${\rm Leb}$ denotes the Lebesgue measure on $\mathbb R^2$. 
Combining \eqref{eq1-leb} and \eqref{eq2-leb} yields \begin{equation}\label{equar3}\# V^*(x,r)\le \frac{81K^2}{\pi r^2 } {\rm Leb}\left(B\left(x, \left(1+\frac{72\sqrt{2}K}{C_1}t^{2 \sqrt{2\log (1/r)}}\right)r\right)\right).\end{equation}
From \eqref{equar} and \eqref{equar3} the desired inequalities follow.
\end{proof}

\begin{proof}[Proof of Proposition~\ref{seed-Prop}]
Let $\mu$ be a Borel probability measure on $I^\infty$ given by Lemma~\ref{unified-lem}.
Let
$\varepsilon\in (0, 1)$.
 Let $r\in(0,r_1)$ and let $x\in\Pi(I^{\infty})$.
By Lemma~\ref{claim0}(b), Lemma~\ref{unified-lem} and Lemma~\ref{claim2}, we have \begin{equation}\label{control3}\begin{split}\mu\circ(\Pi|_{I^\infty})^{-1}((B(x, r))&\le 
 \sum_{\tau\in V^*(x,r)}\mu([\tau])\leq C\sum_{\tau\in V^*(x,r)}{\rm diam}(\phi_{\tau}(\overline{U}))\\
 &\leq C\# V^*(x,r) \max_{\tau\in V^*(x,r)}{\rm diam}(\phi_{\tau}(\overline{U}))\\
 &\leq C  t^{4\sqrt{2\log (1/r)}}r^\varepsilon r^{1-\varepsilon}\leq r^{1-\varepsilon}\end{split}\end{equation}
 for some $C\geq1$. 
The last inequality holds if $r$ is sufficiently small.
The mass distribution principle \cite{Fal14} yields $\dim_{\rm H} \Pi(I^\infty)\ge1-\varepsilon.$ Since $\varepsilon\in(0,1)$ is arbitrary we obtain 
$\dim_{\rm H} \Pi(I^\infty)\ge1$ as required.
\end{proof}

\subsection{Modifying the seed set}\label{sectthmHI}

Let $\varepsilon>0$. Fix a strictly increasing sequence $\{n_k\}_{k=1}^\infty$ of positive integers such that
\begin{equation}\label{H-ineq}t^{n_k+1}>2 t^{n_k}+\sqrt{2\varepsilon n_k} \ \text{ for all }k\geq1,\end{equation}
\begin{equation}\label{H-ineq2}\sum_{j=1}^k \varepsilon n_j(n_j+2)\leq 2\varepsilon n_k^2\ \text{ for all }k\geq1,\end{equation}
 \begin{equation} 
 \label{equanew}
 \prod_{j=1}^{n+1}\frac{1}{(\sqrt{2} t^{j+1})}\cdot \frac{1}{t^{2\varepsilon n^2}}
\geq \left(\prod_{j=1}^n\frac{1}{t^{2(j-1)}}\right)^{1+5\varepsilon}\ \text{ for all }n\geq n_1.
\end{equation}

For each $k\in \mathbb N$, 
define an integer square 
\[W_k=\{2t^{n_k}+2t^{n_k} \sqrt{-1}+x\colon z\in\mathbb Z(\sqrt{-1}),\ {\rm Re}(x)\geq0, {\rm Im}(x)\geq0,\ \|x\|_\infty\leq\sqrt{\varepsilon n_k}-1\},\]
and write $W_k=\{w^{(k)}_1,\ldots,w^{(k)}_{\#W_k}\}$. 
By \eqref{H-ineq} we have
\[\max_{x\in W_k}\|x\|_\infty< 2t^{n_k}+\sqrt{\varepsilon n_k}<t^{n_k+1}<2t^{n_{k+1}}
\leq\min_{x\in W_{k+1}}\|x\|_\infty,\]
which implies that the integer squares are pairwise disjoint. 

Let $y\in S$.
We
 insert the elements of $\bigsqcup_{k=1}^\infty W_{k}$ into $\{c_n(y)\}_{n\in\mathbb N}$ to define a new sequence 
 \[\begin{split}&\ldots,c_{n_k-1}(y),\ c_{n_k}(y),\ \fbox{$w_{1}^{(k)},\ldots,w_{\#W_k }^{(k)}$},\ c_{n_k+1}(y),\ldots\\
&\ldots,c_{n_{k+1}}(y),\ \fbox{$w_{1}^{(k+1)},\ldots,w_{\#W_{k+1}}^{(k+1)}$},\ c_{n_{k+1}+1}(y),\ldots\end{split}\]
Let $x(y)$ denote the point in $U\setminus\mathbb Q^2$ whose Hurwitz continued fraction expansion is given by this new sequence. Let $\widetilde S$ denote the collection of these points:
\[\widetilde S=\{x(y)\in U\setminus\mathbb Q^2\colon y\in S\}.\]

\begin{lem}
\label{hougan1}
$\widetilde S\subset H.$
\end{lem}
\begin{proof}

Let $y\in S$. For each $k\in\mathbb N$ we have
$\Vert c_{n_k}(y)\Vert_\infty <2 t^{n_k}\leq\min_{x\in W_k} \|x\|_\infty$,
and by \eqref{H-ineq},
$\Vert c_{n_k+1}(y)\Vert_\infty\geq t^{n_k+1}>2
t^{n_k}+\sqrt{\varepsilon n_k}>\max_{x\in W_k}\|x\|_\infty.$
Hence we have $c_n(y)\notin W_k$ for all $n\geq1$. 
Since $k\in\mathbb N$ is arbitrary and the integer squares are pairwise disjoint as shown above, 
 $c_m(x(y))\neq c_n(x(y))$ holds for all $m$, $n\in\mathbb N$ with $m\neq n,$
 namely $x(y)\in E.$ 

For any finite set $A\subset\mathbb Z^2$,
there is $k_0\geq1$ such that if $k\geq k_0$ then $W_k$ contains a homothetic copy of $A$.
Hence, $\{c_n(x(y))\colon n\in\mathbb N\}$ contains infinitely many homothetic copies of any finite subset of $\mathbb Z^2$, and so $x(y)\in H$. Since $y\in S$ is arbitrary,  
the desired inclusion holds.
\end{proof}

Clearly the map $y\in S\mapsto x(y)\in \widetilde S$ is bijective.
Let $f\colon  \widetilde S\to S$ denote the inverse of this map, 
which eliminates all the digits in $\bigsqcup_{k=1}^\infty W_k$. 

\begin{pro}
\label{proholder}
There exists $C>0$ such that  
\[|f(x_1)-f(x_2)|\leq C|x_1-x_2|^{\frac{1}{1+5\varepsilon}}\]
for all $x_1$, $x_2\in \widetilde S$.
\end{pro}
We finish the proof of Theorem~\ref{thmHI} assuming Proposition~\ref{proholder}.
Since $|c_n(x)|\to \infty$ as $n\to \infty$ for all $x\in E$, from \cite[Theorem~1.3]{Ger} it follows that $\dim_{\rm H} E\le 1$. By $H\subset E$ and Lemma \ref{hougan1}, we have $\dim_{\rm H} E\ge\dim_{\rm H} H\ge \dim_{\rm H}\widetilde S.$
By 
\cite[Proposition~3.3]{Fal14} and 
Proposition~\ref{proholder}, we have \[\dim_{\rm H}\widetilde S\ge  \frac{1}{1+5\varepsilon}\dim_{\rm H} S=\frac{1}{1+5\varepsilon}.\]
Since $\varepsilon>0$ is arbitrary we obtain $\dim_{\rm H} H\ge 1,$ namely  $\dim_{\rm H}E=\dim_{\rm H}H=1.$
\qed

\if0\begin{lem}
\label{lemesti1}
For all $q\ge 1,$
\[\prod_{j=1}^q\prod_{i\in W_j}\frac{1}{|i|+1}\geq \frac{1}{t^{2\varepsilon (n_q^2+ n_q)}}.\]
\end{lem}
\begin{proof}
For all $1\leq j\leq q,$ we have 
\[\prod_{i\in W_j}\frac{1}{|i|+1}\geq \prod_{i\in W_j}\frac{1}{\sqrt{2}(2t^{n_j}+\varepsilon \sqrt{n_j})+1}\ge  \frac{1}{(\sqrt{8}t^{n_j}+\sqrt{2}\varepsilon \sqrt{n_j}+1)^{\varepsilon n_j}}.\]
Since $\sqrt{8}t^{n_j}+\sqrt{2}\varepsilon \sqrt{n_j}+1\leq \sqrt{8}t^{n_j}+t^{n_j}\leq t^{(n_{j}+1)}$ by $4\leq  t,$ we have
\[\prod_{j=1}^q\prod_{i\in W_j}\frac{1}{|i|+1}\geq \prod_{j=1}^q \frac{1}{t^{{\varepsilon n_j}(n_{j}+1)}}.\]
Since $n_q\gg n_i$ for all $i=1,\ldots, q-1$, we have \[\sum_{j=1}^q \varepsilon n_j(n_j+1)\leq 2\varepsilon (n_q^2+ n_q),\] which implies the desired inequality.
\end{proof}
\fi

\medskip

An idea behind a proof of
Proposition~\ref{proholder} is that 
$|c_1(y)|,\ldots,|c_{n_q}(y)|$ kills
the effect of modification by the integer squares $W_1,\ldots,W_q$ evaluated in the next lemma. 
\begin{lem}
\label{lemesti1}
For all $q\ge 1$ we have
\[\prod_{k=1}^q\prod_{i\in W_k}\frac{1}{|i|+1}\geq \frac{1}
{t^{2\varepsilon n_q^2}}.\]
\end{lem}
\begin{proof}
For all $1\leq k\leq q$ we have $\#W_k\leq\varepsilon n_k$ and
\[\prod_{i\in W_k}\frac{1}{|i|+1}> \frac{1}{(\sqrt{2}(2 t^{n_k}+\sqrt{\varepsilon n_k})+1)^{\#W_k}}\ge  \frac{1}{(\sqrt{8}  t^{n_k}+\sqrt{2\varepsilon n_k}+1)^{\varepsilon n_k}}.\]
 By \eqref{H-ineq} and $t=3$ we have
 $\sqrt{8} t^{n_k}+\sqrt{2\varepsilon n_k}+1\leq \sqrt{8} t^{n_k}+t^{n_k}\leq t^{n_{k}+2}$. Plugging this inequality into the above denominator, multiplying the result over all $1\leq k\leq q$ and then using \eqref{H-ineq2} we obtain
\[\prod_{k=1}^q\prod_{i\in W_k}\frac{1}{|i|+1}\geq \prod_{k=1}^q \frac{1}{t^{{\varepsilon n_k}(n_{k}+2)}}\geq \frac{1}{t^{2\varepsilon n_q^2}}.\qedhere\]
\end{proof}

\begin{proof}[Proof of Proposition~\ref{proholder}]
In view of the isolation property (V), take $\delta>0$ such that $\overline U$ contains the $\delta$-neighborhood of $\Lambda$. 
For $y_1$, $y_2\in S$ with $y_1\neq y_2$,
let $s(y_1,y_2)$ denote the minimal  integer $n\geq0$ such that $c_{n+1}(y_1)\neq c_{n+1}(y_2).$
Notice that \begin{equation}\label{inf-ineq}\inf\{|f^{-1}(y_1)-f^{-1}(y_2)|\colon y_1,y_2\in S,\ y_1\neq y_2,\ s(y_1,y_2)<n_1\}>0.\end{equation}

Let $x_1, x_2\in \widetilde S$, $x_1\neq x_2$ and put
 $y_1=f(x_1)$, $y_2=f(x_2)$. 
Suppose $s(y_1,y_2)=n\geq n_k$ for some $k\geq1$.
For each $q\in\mathbb N$, set
 $m_q=\#W_q$.
 Let $q\geq k$ be the integer such that $n_q\le n< n_{q+1}.$ We have 
$c_j(x_1)=c_j(x_2)$ for $1\le j\le n+m_ {1}+\cdots+m_{q}$,
and  
$c_{n+m_1+\cdots+m_q+1}(x_1)=c_{n+1}(y_1)\neq c_{n+1}(y_2)=c_{n+m_1+\cdots+m_q+1}(x_2).$
Applying Lemma~\ref{conf-lem} with $\phi=\phi_{c_1(x_1)c_2(x_1)\cdots c_{n+m_1+\cdots+m_q+1}(x_1)}$ we obtain
\[\label{dis-bound}x_2\notin B\left(x_1,\frac{\delta|D\phi(\phi^{-1}(x_1))|}{3K} \right)\subset \phi(B(\phi^{-1}(x_1),\delta)).
\]
By Lemma~\ref{diam-lem}, Lemma~\ref{lemesti1}, $|c_j(y_1)| \leq \sqrt{8} t^j-1$ for $1\leq j\leq n+1$ and  finally \eqref{equanew}, for some $K\geq1$ we have
\[\label{holderine3}\begin{split}|x_1-x_2|&\ge \frac{\delta|D\phi(\phi^{-1}(x_1))|}{3K} \geq \frac{\delta}{3\sqrt{2}K^2 }{\rm diam}(\phi (\overline{U}))\\
&\geq C\frac{1}{\left((|c_1(x_1)|+1)\cdots (|c_{n+m_1+\cdots+m_q+1}(x_1)|+1)\right)^2}\\
&= C\prod_{j=1}^{n+1}\frac{1}{(|c_j(y_1)|+1)^2}\cdot\prod_{k=1}^q\prod_{i\in W_k}\frac{1}{(|i|+1)^2}\\
&\geq C \prod_{j=1}^{n+1}\frac{1}{(\sqrt{8} t^{j})^2}\cdot \frac{1}{t^{4\varepsilon n_q^2  }}\geq  C\prod_{j=1}^{n+1}\frac{1}{(\sqrt{2} t^{j+1})^2}\cdot \frac{1}{t^{4\varepsilon n^2 }}\\&\geq  C\left(\prod_{j=1}^{n}\frac{1}{t^{2(j-1)}}\right)^{2(1+5\varepsilon)},\end{split}\]
where $C=\delta\gamma/(3\sqrt{2}K^2).$
Meanwhile,
the second inequality in Lemma~\ref{diam-lem} and $|c_j(y_1)|\geq t^j$ for $1\leq j\leq n$ gives
\[\label{holder-ineq4}
  |y_1-y_2|\leq {\rm diam}(\phi_{c_1(y_1)\cdots c_n(y_1)}(\overline{U}))\leq
  2\prod_{j=1}^n\frac{1}{(|c_j(y_1)|-1)^2}
\leq2\prod_{j=1}^n\frac{1}{ t^{2(j-1)}}.\]
Combining these two estimates yields
\[|x_1-x_2|\ge C\frac{|y_1-y_2|^{1+5\varepsilon}}{2^{1+5\varepsilon}} \geq\frac{C}{64}|y_1-y_2|^{1+5\varepsilon}.\]
Together with \eqref{inf-ineq} we obtain
the desired inequality in Proposition~\ref{proholder}.
\end{proof}

\subsection{Final remark} 

An interesting generalization of Theorem~\ref{thmHI} is to fix an infinite proper subset $S$ of $\mathbb Z^2$, and compute the Hausdorff dimension of the set of complex irrationals in $U$ whose partial quotients contain infinitely many homothetic copies of any finite subset of $\mathbb Z^2$ that are contained in $S$. The forthcoming paper \cite{NT4} treats this topic.

\subsection*{Acknowledgments} 
YN was supported by the JSPS KAKENHI 25K17282, Grant-in-Aid for Early-Career Scientists. HT was supported by the
JSPS KAKENHI 25K21999, Grant-in-Aid for Challenging Research (Exploratory).

\end{document}